\documentclass{amsart}
\usepackage{amsfonts,amssymb,stmaryrd,amscd,amsmath,latexsym,amsbsy}
\numberwithin{equation}{section}
\theoremstyle{plain}
\newtheorem{thm}{Theorem}[section]

\newtheorem{prop}[thm]{Proposition}

\newtheorem{lem}[thm]{Lemma}
\newtheorem{cor}[thm]{Corollary}

\newtheorem{eg}[thm]{{Example}}

\theoremstyle{remark}
\newtheorem{rema}[thm]{Remark}

\newcommand{\al}{\alpha}

\newcommand{\N}{{\mathbb N}}

\newcommand{\field}{{k}}

\newcommand{\gfrak}{{\mathfrak g}}

\newcommand{\id}{{\mbox{id}}}

\newcommand{\kow}{{\varDelta}}

\newcommand{\ot}{\otimes}

\newcommand{\rwo }{\leq _R}

\newcommand{\kspan}{\mathrm{span}}

\newcommand{\uqbp}{{U^{\ge 0}}}

\newcommand{\uqg}{{U_q(\mathfrak{g})}}

\newcommand{\vep}{\varepsilon}

\newcommand{\wurz}{\Phi}

\newcommand{\Z}{{\mathbb Z}}

\begin{document}

\title[Homogeneous right coideal subalgebras]{Homogeneous
right coideal subalgebras of quantized enveloping algebras}
\thanks{The work of I. Heckenberger was supported
by a Heisenberg professorship of the German Research Foundation (DFG)}

\author{I.~Heckenberger}
\address{Istvan Heckenberger,
Fachbereich Mathematik und Informatik, Philipps-Universit\"at
Marburg, Hans-Meerwein-Stra\ss e, D-35032 Marburg, Germany}
\email{heckenberger@mathematik.uni-marburg.de}
\author{S.~Kolb}
\address{Stefan Kolb, School of Mathematics and Statistics, Newcastle University, Newcastle upon Tyne NE1 7RU, UK}
\email{stefan.kolb@ncl.ac.uk}


\begin{abstract}
  For a quantized enveloping algebra of a complex semisimple Lie algebra with deformation parameter not a root of unity, we classify all homogeneous right coideal subalgebras. 
  Any such right coideal subalgebra is determined uniquely by a triple
  consisting of two elements of the Weyl group and a subset of the set of
  simple roots satisfying some natural conditions. The essential ingredients
  of the proof are the Lusztig automorphisms and the classification of
  homogeneous right coideal subalgebras of the Borel Hopf subalgebras of
  quantized enveloping algebras obtained previously by H.-J.~Schneider and the first named author.
\end{abstract}

\maketitle

\section{Introduction}

Unlike their classical analogs, quantized enveloping algebras do not admit
many Hopf subalgebras. Folklore has it that Lie subalgebras of a, say,
semisimple complex Lie algebra $\gfrak$ should possess quantum analogs which
are coideal subalgebras of the corresponding quantized enveloping algebra
$\uqg$. Once this point of view was established in the early nineties, large
classes of quantum group analogs of homogeneous spaces were constructed via
(one-sided or two-sided) coideal subalgebras of $\uqg$ \cite{a-NS95},
\cite{a-Dijk96}, \cite{MSRI-Letzter}. All these constructions, however, are
closely tied to the underlying classical situation.

In the present paper we follow a different route which aims at a general
classification of right coideal subalgebras of $\uqg$ under mild additional
assumptions. More precisely, let $U^0$ denote the subalgebra generated by all
group-like standard generators of $\uqg$.
We call a right coideal subalgebra of $\uqg$
\textit{homogeneous}
if it contains $U^0$. In the present paper we determine all homogeneous right coideal subalgebras of
$\uqg$, for $q$ not a root of unity, in terms of pairs of elements of the
corresponding Weyl group $W$. Our main result, Theorem \ref{thm:rcsaU}, states the
following:

\medskip

\textit{There exists a one-to-one correspondence between the set of
homogeneous right coideal subalgebras of $\uqg$ and the set of triples
$(x,u,J)$, where $x,u\in W$, the element $u^{-1}$ is a beginning of $x$, and
$J$ is an arbitrary subset of $\Pi\cap x \Pi$, where $\Pi$ denotes a fixed
basis of the root system of $\gfrak$.}

\medskip

Here we call $u^{-1}$ a beginning of $x$ if $u^{-1}\le_R x$ with respect to the weak (or Duflo) order $\le_R$ on $W$, see Section \ref{sec:uqbp} for a precise definition. 

Our theorem implies in particular that the number $|B(W)|$ of homogeneous right
coideal subalgebras of $\uqg$ depends only on $W$ and not explicitly on
$\gfrak$. With the help of the computer algebra program FELIX
\cite{inp-ApelKlaus,a-FELIX} we determine the numbers $|B(W)|$ for all $W$ of
order less than $10^6$. The largest such example is the Weyl group of type
$B_7$ which has order $645120$. Our calculations confirm the results obtained previously
by Kharchenko and Sagahon in \cite{a-KharSag08} for type $A_n$.
However, our numbers differ from those given recently by Pogorelsky
in type $G_2$ \cite {a-Pogorelsky11} and by Kharchenko, Sagahon,
and Rivera in type $B_n$ with $n\ge 3$ \cite{a-KharSagRiv11}. We comment on these differences
in Remark~\ref{rem:diffnumbers}.

The systematic investigation of homogeneous right coideal subalgebras of $\uqg$ was initiated by
Kharchenko in \cite{a-KharSag08}. This led to the astonishing conjecture that the number of
homogeneous right coideal subalgebras of the Borel Hopf subalgebra $\uqbp $ of $\uqg$ coincides
with the order of the Weyl group $W$ of $\gfrak$. This conjecture provided the
first indication that coideal subalgebras of $\uqg$ can be classified in terms
of Weyl group combinatorics. It was proved for $\gfrak$ of type $A_n$ and
$B_n$ by Kharchenko \cite{a-KharSag08}, \cite{a-Kharchenko09}.

Recently, Kharchenko's conjecture was proved in a much wider context  by H.-J.~Schneider
and the first named author, namely for bosonizations
$\mathcal{B}(V)\#H$ of Nichols algebras of semisimple Yetter-Drinfeld modules $V$
over an arbitrary Hopf algebra $H$ with bijective antipode
\cite{a-HeckSchn09p}. For $\uqg$ they showed that the homogeneous right
coideal subalgebras of $\uqbp $ are precisely the algebras $U^+[w]U^0$,
where $U^+[w]$ denotes the quantum analog of a nilpotent Lie algebra
introduced in \cite{a-DCoKacPro95} for any $w\in W$. The algebras $U^+[w]$
have recently also undergone a renaissance within the general program of
investigation of prime spectra of quantum algebras, see for example
\cite{a-Yakimov10}. It is noteworthy that in the classification of prime
ideals of $U^+[w]$ which are invariant under the adjoint action of $U^0$, the
Bruhat (or strong) order plays an important role while the poset structure of
homogeneous right coideal subalgebras of $\uqg$ given by inclusion is identified with the weak order of $W$.

To describe homogeneous right coideal subalgebras of $\uqg$ one uses the
triangular decomposition of $\uqg$. One observes that any homogeneous right
coideal subalgebra $C$ of $\uqg$ can be written as a product $C=C^-C^+$ where
$C^-$ and $C^+$ are homogeneous right coideal subalgebras of the negative and
the positive Borel Hopf subalgebras $U^{\le 0}$ and $\uqbp $, respectively.
As the homogeneous right coideal subalgebras of $U^{\le 0}$ and $\uqbp $
are known by the classification described above, it remains to determine when
their product is a subalgebra. Again, Kharchenko et al.~\cite{a-KharSag08},
\cite{a-KharSagRiv11} 
were able to handle the series $A_n$ and $B_n$.
However, their approach does not
exhibit any obvious relation to the combinatorics of Weyl groups.

The classification of homogeneous right coideal subalgebras of $\uqg$
presented in this paper is certainly not the end of the story. In
\cite{a-heko11} we classified all right coideal subalgebras $C$ of $\uqbp $
for which $C\cap U^0$ is a Hopf algebra. In view of the results of the present
paper it now seems feasible to classify right coideal
subalgebras $C$ of $\uqg$ under the weaker assumption that $C\cap U^0$ is a
Hopf algebra. This class contains in particular all quantum symmetric pair
coideal subalgebras defined in \cite{MSRI-Letzter}. We expect that such a
classification will conceptualize and simplify the understanding of many
existing examples of right coideal subalgebras of $\uqg$.

\section{Preliminaries}
\subsection{Quantized enveloping algebras}
We mostly follow the notation and conventions of \cite{b-Jantzen96}.
Let $\gfrak $ be a finite-dimensional complex semisimple Lie algebra and let
$\wurz $ be the root system with respect to a fixed Cartan
subalgebra. We also fix a basis $\Pi$ of  $\Phi$ and denote by
$\wurz^+$ the corresponding set of positive roots.
Let $W$ be the Weyl group of $\gfrak $ and let $(\cdot ,\cdot )$ be
the invariant scalar product on the real vector space generated by
$\Pi $ such that $(\al ,\al )=2$ for all short roots in each
component. For any $\beta\in \wurz$ we write $s_\beta$ to denote the
reflection at the hyperplane orthogonal to $\beta$ with respect to
$(\cdot, \cdot)$.
Let $Q=\Z \Pi $ be the root lattice . 
Let $U=\uqg $ be the quantized enveloping algebra of $\gfrak $
in the sense of \cite[Chapter 4]{b-Jantzen96}.
More precisely, let $\field $ be a field and
fix an element $q\in \field $ with $q\not=0$ and $q^n\not=1$ for all
$n\in \N $. Then $U$ is the
unital associative algebra defined over $\field $
with generators $K_\al , K_\al ^{-1}, E_\al , F_\al $ for all $\al \in \Pi
$ and relations given in \cite[4.3]{b-Jantzen96}. 
By \cite[Proposition 4.11]{b-Jantzen96}
there is a unique Hopf algebra structure on $U$ with coproduct $\kow$,
counit $\vep$, and antipode $S$ such that
\begin{align}
  \label{eq:UHopf1}
  \kow (E_\al )=&E_\al \ot 1+K_\al \ot E_\al ,&
  \vep (E_\al )=&0,& S(E_\al)&=-K_\al^{-1} E_\al ,\\
  \kow (F_\al )=&F_\al \ot K_\al ^{-1}+1 \ot F_\al ,&
  \vep (F_\al )=&0,& S(F_\al)&=-F_\al K_\al,\label{eq:UHopf2}\\
  \kow (K_\al )=&K_\al \ot K_\al ,& \vep (K_\al )=&1, &S(K_\al)&=K_\al^{-1}.
  \label{eq:UHopf3}
\end{align}
As in \cite[Chapter 4]{b-Jantzen96} let $U^+$, $U^-$, and $U^0$
be the subalgebras of $U$ generated by the sets $\{E_\al \,|\,\al \in \Pi
\}$, $\{F_\al \,|\,\al \in \Pi
\}$, and  $\{K_\al ,K_\al ^{-1}\,|\,\al \in \Pi \}$, respectively. Moreover,
let $\uqbp $ and $U^{\le 0}$ be the subalgebras of $U$ generated by the
sets $\{ K_\al ,K_\al ^{-1},E_\al \,|\,\al \in \Pi \}$ and
$\{ K_\al ,K_\al ^{-1},F_\al \,|\,\al \in \Pi \}$, respectively.
In fact, $U^0$ and $\uqbp$ and $U^{\le 0}$ are Hopf subalgebras of $\uqg $.
For any $\beta \in Q$ define $K_\beta =\prod _{\al \in
\Pi }K_\al ^{n_\al }$ if $\beta =\sum _{\al \in \Pi }n_\al \al $ for
some $n_\alpha\in \Z$. 

For any subspace $M$ of $U$ which is invariant under conjugation by all $K_\al$ for $\al \in \Pi$, we define the weight space $M_\beta$ of weight $\beta\in Q$ by
\begin{align*}
  M_\beta=\{m\in M\,|\, K_\al m K_\al^{-1}= q^{(\al,\beta)}m\mbox{ for all $\al\in \Pi$} \}.
\end{align*}
We will apply this notation in particular if $M$ is one of $U, U^+,$ or $U^-$. We call the elements of $M_\beta$ weight vectors of weight $\beta$ in $M$.
For two subspaces $A,B$ of $\uqg $ we write $AB$ for the subspace
$\kspan_\field \{ab\,|\,a\in A,b\in B\}$. For a subset $A\subseteq \uqg $
we write $\field \langle A\rangle $ for the unital subalgebra of $\uqg $ generated by $A$.
\subsection{The triangular decomposition}
The algebra $U$ has a triangular decomposition in the sense that the multiplication map
\begin{align}\label{eq:triang}
  U^-\ot U^0 \ot U^+ \rightarrow U 
\end{align}
is an isomorphism of vector spaces. Recall from the introduction that we call a right coideal subalgebra of $U$ homogeneous if it contains $U^0$. The triangular decomposition of $U$ descends to the level of homogeneous right coideal subalgebras of $U$. This result is in principle contained in \cite[Section 4]{MSRI-Letzter} and was also proved in \cite{a-Kharchenko10}. Here we recall the argument for the convenience of the reader.
\begin{prop}\label{prop:triang}
  Let $C$ be a homogeneous right coideal subalgebra of $U$. Then the multiplication map
  \begin{align}\label{eq:Ctriang}
    (U^-\cap C)\ot U^0 \ot (U^+\cap C) \rightarrow C
  \end{align}
  is an isomorphism of vector spaces.
\end{prop}
\begin{proof}
  In view of the triangular decomposition \eqref{eq:triang} of $U$
  it suffices to show that the multiplication map \eqref{eq:Ctriang}
  is surjective.
  To this end consider any weight vector $x\in C_\beta$ of weight $\beta\in Q$.
  Using the triangular decomposition \eqref{eq:triang} one can write
  \begin{align*}
    x=\sum_{\mu, \alpha, i} F_{\beta-\al,\mu,i}K_\mu E_{\al,\mu,i}
  \end{align*}
  where for fixed $\mu,\al\in Q$, the elements of $\{E_{\alpha,\mu,i}\}$ are linearly independent weight vectors of weight $\alpha$ in $U^+$, and similarly, the elements of $\{F_{\beta-\alpha,\mu,i}\}$ are linearly independent weight vectors of weight $\beta-\alpha$ in $U^-$.
  The relations \eqref{eq:UHopf1} - \eqref{eq:UHopf3} imply that for any $\alpha, \mu, i$ one has
\begin{align}\label{eq:uptohigher}
  \kow(F_{\beta-\al,\mu,i}K_\mu E_{\al,\mu,i})\in K_\mu E_{\al,\mu,i}\ot
  F_{\beta-\al,\mu,i}K_\mu  + \sum_{\gamma<\al} U_{\gamma} \ot U_{\beta-\gamma}.
\end{align}
Now choose $\alpha\in Q$ maximal such that
$E_{\alpha,\mu,i}\neq 0\neq F_{\beta-\alpha,\mu,i}$ for some $\mu,i$.
By \eqref{eq:uptohigher}, the linear independence of the set
$\{F_{\beta-\alpha,\mu ,i}K_\mu \,|\,\mu\in Q\}$ and the maximality
of $\al$ imply that $K_\mu E_{\al, \mu,i}\in C$
for all $\mu,i$. As $U^0\subseteq C$ one gets $E_{\al, \mu,i}\in C$. A similar argument proves that $F_{\beta-\alpha,\mu,i}\in C$ for all $\mu,i$.
Now one obtains
$$x\in (U^-\cap C)U^0(U^+\cap C) $$
by induction on $\alpha$.
\end{proof}
\subsection{Homogeneous right coideal subalgebras of $\uqbp $}\label{sec:uqbp}
For any right coideal subalgebra $C$ of $U$ with $U^0\subseteq C$,
Proposition~\ref{prop:triang} implies that $C=C^- C^+$
where $C^+=(C\cap U^+)U^0$ and $C^-=(C\cap U^-)U^0$
are right coideal subalgebras of $\uqbp $ and $U^{\le 0}$, respectively,
both containing $U^0$. Such coideal subalgebras of $\uqbp $
were classified in \cite{a-HeckSchn09p}. In order to formulate this result,
we first recall the Lusztig automorphisms $T_\alpha$
of $U$ defined for $\al\in \Pi$ in \cite[8.14]{b-Jantzen96}
and the quantum analogs $U^+[w]$ of nilpotent Lie algebras defined for
any $w\in W$. 

Let $w\in W$ be an element of length $\ell (w)=t$ and choose $\al
_1,\dots,\al _t\in \Pi $ such that $s_{\al _1}s_{\al _2}\cdots s_{\al
  _t}$ is a reduced expression of $w$. Until the end of this section we will keep these notations whenever we are in need of a reduced expression of an element $w\in W$.
  For all $i\in \{1,2,\dots,t\}$
let $\beta _i=s_{\al _1}\cdots s_{\al _{i-1}}\al _i$. We define
\begin{align*}
  \Phi^+(w)=\{\al\in \Phi^+\,|\,w^{-1}\al<0\}
\end{align*}
and one verifies that $\Phi^+(w)=\{\beta_i\,|\, i=1,\dots,t\}$.
Moreover, one defines an automorphism $T_w=T_{\al _1}T_{\al _2} \cdots T_{\al _t}$ which by \cite[8.18]{b-Jantzen96} is independent of the chosen reduced expression for $w$. The inverse of $T_w$ is obtained by
\begin{align}\label{eq:Tw-1}
  T_w^{-1}=\tau \circ T_{w^{-1}}\circ \tau
\end{align}
where $\tau$ is the involutive algebra antiautomorphism of $U$ determined by
\begin{align*}
  \tau(E_\al)=E_\al,\quad \tau(F_\al)=F_\al, \quad \tau(K_\al)=K_\al^{-1},\quad \tau(K_\al^{-1})=K_\al,
\end{align*}
for all $\al \in \Pi$, see \cite[8.18(6)]{b-Jantzen96}. The following Lemma will be useful to characterize right coideal subalgebras of $U^{\le 0}$ in the next subsection.
\begin{lem} \label{le:S-tau}
  Let $\mu =\sum_{\alpha\in \Pi} n_\alpha \alpha \in Q$ and $u\in U^-_{-\mu }$. Then 
  $ \tau(u)= c(\mu) S(u) K_{-\mu}$
  where
  $$c(\mu)=(-1)^{\sum_{\al\in\Pi}n_\alpha} q^{\frac{1}{2}(\mu ,\mu )-
  \frac{1}{2}\sum_{\al\in \Pi} n_{\al}(\al,\al) }$$
  only depends on $\mu$ and not on $u$.
\end{lem}
\begin{proof}
  The lemma holds for $\mu =0$ and for $\mu \in \Pi $ by definition of $\tau
  $.
  Thus it holds for all $\mu $ since $\tau $ is an algebra antiautomorphism
  and
  \begin{align*}
    c(\mu _1)c(\mu _2)S(u_2)K_{-\mu _2}S(u_1)K_{-\mu _1}=&\;
    c(\mu _1)c(\mu _2)q^{(\mu _1,\mu _2)}S(u_2)S(u_1)K_{-\mu _1-\mu _2}\\
    =&\;c(\mu _1+\mu _2)S(u_1u_2)K_{-\mu _1-\mu _2}
  \end{align*}
  for all $\mu _1,\mu _2\in Q$ and $u_1\in U^-_{-\mu _1}$, $u_2\in U^-_{-\mu
  _2}$.
\end{proof}
In \cite[8.21]{b-Jantzen96} Lusztig root
vectors in $U^+$ are defined by $E_{\beta _i}=T_{\al_1}\cdots
T_{\al_{i-1}}E_{\al_i}$ for all  $i\in \{1,\dots,t\}$.
It is known \cite[8.21(1)]{b-Jantzen96} that $E_{\beta _i}\in U^+_{\beta _i}$
for all $i\in \{1,\dots,t\}$. Moreover, if $\beta _i\in \Pi $
for some $i$ then $E_{\beta _i}$ coincides with the standard generator
$E_{\beta _i}$ of $U^+$ by \cite[8.20]{b-Jantzen96}.
Following \cite{a-DCoKacPro95} the subspace 
\begin{align}\label{eq:U+wdef}
 U^+[w]=\kspan _\field \{ E_{\beta _t}^{a_t} \cdots E_{\beta _2}^{a_2}
  E_{\beta _1}^{a_1} \,|\,a_1,\dots,a_t\in \N _0 \}
\end{align}  
is attached to $w$ in \cite[8.24]{b-Jantzen96}. It is shown in \cite{a-DCoKacPro95} that $U^+[w]$ is a subalgebra which does not depend on the reduced expression of $w$. For later purposes we recall the PBW-theorem for $U^+[w]$, see \cite[Chapter 8]{b-Jantzen96}.
\begin{prop} \label{prop:PBW+} 
The set
\[
\{ E_{\beta _1}^{n_1} E_{\beta _2}^{n_2} \cdots E_{\beta _t}^{n_t}\,|\, n_1,n_2,\dots,n_t\in \N_0
\}
\]
is a basis of $U^+[w]$.
\end{prop}
  Definition \eqref{eq:U+wdef} and Proposition \ref{prop:PBW+} imply the following result.
\begin{cor} \label{cor:U+wdec}
 Let $x,y\in W$ with $\ell (xy)=\ell (x)+\ell (y)$. Then
 $$U^+[xy]=U^+[x]T_x(U^+[y])=T_x(U^+[y]) U^+[x].$$
\end{cor}
The algebras $U^+[w]$ for $w\in W$ are significant building blocks of right
coideal subalgebras of $\uqbp $. To make this statement precise, recall that
the Weyl group is a poset with respect to the
weak (or Duflo) order which we denote by $\le_R$, see \cite[Chapter 3]{b-BjornerBrenti06}. By definition, one has $v\le_R w$ for $v,w\in W$ if there exists $u\in W$ such that $w=vu$ and
$\ell(w)=\ell(v)+\ell(u)$. On the other hand, the set of right coideal
subalgebras of $U$ which contain $U^0$ is also a poset with respect to
inclusion. The following theorem, which was proved in \cite[Theorem
7.3]{a-HeckSchn09p}, now explains the role of the algebras $U^+[w]$ in the
investigation of homogeneous right coideal subalgebras of $U$.
\begin{thm}\label{thm:rcsa+} 
  The map from $W$ to the set of right coideal subalgebras of
  $\uqbp $ containing $U^0$, given by $w\mapsto U^+[w]U^0$, is
  an order preserving bijection.
\end{thm}
As a first application of the above theorem we prove a Lemma which will be repeatedly used in Section \ref{sec:compatible}.
\begin{lem} \label{lem:EUw}
  Let $w\in W$ and $\alpha \in \Pi $.

  \begin{enumerate}
    \item The following are equivalent.
      \begin{enumerate}
        \item $E_\alpha \in U^+[w]U^0$,
        \item $\ell (s_\alpha w)=\ell (w)-1$,
        \item $\alpha \in \Phi ^+(w)$.
      \end{enumerate}
    \item Assume that $\ell (s_\alpha w)=\ell (w)+1$. The following are
      equivalent.
      \begin{enumerate}
        \item The subspace $\field \langle E_\alpha \rangle U^+[w]U^0$
          of $\uqbp $ is a subalgebra,
        \item $\alpha \in \Pi \cap w\Pi $,
        \item There exists $\beta \in \Pi $ such that $s_\alpha w=ws_\beta $.
      \end{enumerate}
  \end{enumerate}
\end{lem}
\begin{proof}
  (1) Property (b) implies (a) as $U^+[w]$ is independent of the chosen reduced expression for $w$. Moreover (b) implies (c) by definition of $\Phi ^+(w)$. If $\ell (s_\alpha w)=\ell (w)+1$, on the other  hand, then
  $\Phi ^+(s_\alpha w)=\{\alpha \}\cup s_\alpha (\Phi ^+(w))$.
  Since $\Phi ^+(s_\alpha w)\subseteq \Phi ^+$, we conclude that
  $\alpha \notin \Phi ^+(w)$. Consequently, $E_\alpha \notin U^+[w]U^0$
  since $U^+[w]U^0$ is $\Z \Pi $-graded.

  (2) Clearly, (b) is equivalent to (c) with $\beta =w^{-1}\alpha $,
  and (c) implies (a) by Corollary~\ref{cor:U+wdec} since $w\beta =\alpha $.
  Assume now that (a) holds. By Theorem~\ref{thm:rcsa+},
  since $\field \langle E_\alpha \rangle U^0$ is a right coideal of
  $\uqbp $, there exists $v\in W$ such that
  $\field \langle E_\alpha \rangle U^+[w]U^0=U^+[v]U^0$.
  Theorem~\ref{thm:rcsa+} and $U^+[w]U^0\subseteq U^+[v]U^0$ imply that
  $w\rwo v$. Proposition~\ref{prop:PBW+} yields now that $v=ws_\beta $
  with $\beta \in \Pi $, $w\beta =\alpha $, and hence (b) holds.
\end{proof}
\subsection{Homogeneous right coideal subalgebras of $U^{\le 0}$}
As a consequence of Theorem \ref{thm:rcsa+} we can immediately describe all homogeneous right coideal subalgebras of $U^{\le 0}$. To this end consider the algebra automorphism $\omega$ of $U$ given by
\begin{align*}
  \omega(E_\alpha )=F_\alpha ,\quad \omega(F_\alpha )=E_\alpha ,\quad \omega(K_\alpha )=K_\alpha ^{-1}
\end{align*}
for all $\alpha \in \Phi $. Following \cite[8.24]{b-Jantzen96} one defines $U^-[w]=\omega (U^+[w])$.
Since $\omega$ is a coalgebra antiautomorphism, one concludes that $U^-[w]U^0=\omega(U^+[w]U^0)$ is a left coideal subalgebra of $U$. To obtain a right coideal subalgebra we consider $S(U^-[w]U^0)$. Theorem \ref{thm:rcsa+} implies the following result.
  \begin{cor}\label{cor:rcsa-}
    Any right coideal subalgebra of $U^{\le 0}$ which contains $U^0$ is of the form $S(U^-[w])U^0$ for some $w\in W$.
  \end{cor}
In the remainder of this subsection we use the PBW-theorem for $U^+[w]$ and properties of the Lusztig automorphisms to obtain a suitable PBW-basis of $S(U^-[w])U^0$. By \cite[Eq.~8.14(9)]{b-Jantzen96} one has
\begin{align}
  T_\alpha (\omega (u))\in \field ^\times \omega (T_\alpha (u))
  \label{eq:Tomega}
\end{align}
for all $\alpha \in \Pi $ and all weight vectors $u\in U^+$, where $\field ^\times =\field \setminus \{0\}$. In view of Proposition \ref{prop:PBW+} this implies that $U^-[w]$ has a PBW-basis
\begin{align}\label{eq:Uminus-PBW}
  \{F_{\beta _1}^{n_1} F_{\beta _2}^{n_2} \dots F_{\beta _t}^{n_t}\,|\,n_1,\dots,n_t\in \N _0\}
\end{align}
where $F_{\beta _i}=T_{\alpha _1}\cdots T_{\alpha _{i-1}}F_{\alpha _i}$ for all
$i\in \{1,2,\dots,t\}$.
The following lemma will allow us to describe a PBW-basis of $S(U^-[w])U^0$ without explicit use of the antipode.
\begin{lem}\label{lem:S-rewrite}
  For any $w\in W$ one has
  \begin{align*}
    T_w^{-1} \left(U^+[w] U^0\right) = S(U^-[w^{-1}])U^0.
  \end{align*}
\end{lem}
\begin{proof}
 We perform induction on the length $\ell (w)$ of $w$. If $\ell (w)=0$ then both sides of the above
 expression coincide with $U^0$. Now assume that there exists $\al \in \Pi$ such that $w=s_\alpha v$ with $\ell (w)=\ell (v)+1$. Then one calculates
 \begin{align*}
   T_w^{-1}\left(U^+[w]U^0\right)=&\;
   (T_\alpha T_v)^{-1}\left( \field \langle E_\alpha \rangle T_\alpha (U^+[v]) U^0\right)\\
    =&\;T_v^{-1}\left(\field \langle F_\alpha \rangle U^+[v]U^0\right).
 \end{align*} 
 By induction hypothesis this implies the relation
 \begin{align}\label{eq:stepone}
    T_w^{-1}\left(U^+[w]U^0\right)&=
    \field \langle T_v^{-1}F_\alpha \rangle S(U^-[v^{-1}])U^0.
 \end{align}
 In view of the PBW-basis of $U^-[w^{-1}]$, see \eqref{eq:Uminus-PBW},
 it remains to show that $T_v^{-1}F_\alpha \in S (T_{v^{-1}}F_\alpha )U^0$.
 Now, using \eqref{eq:Tw-1}, we obtain that
\begin{align*}
  T_v^{-1}F_\alpha =&\;\tau \circ T_{v^{-1}}\circ \tau(F_\alpha )\\
  =&\;\tau \circ T_{v^{-1}}F_\alpha \in S(T_{v^{-1}}F_\alpha )U^0,
\end{align*}
where the last relation holds by Lemma~\ref{le:S-tau}.
This proves the lemma.
\end{proof}
Recall that we have fixed a reduced expression $w=s_{\al_1} s_{\al_2} \dots
s_{\al_t}$ of the element $w\in W$. We can now verify the desired PBW-theorem.
\begin{cor}\label{cor:PBW}
  The set
  \begin{align*}
    \{(T_{\al_1}^{-1} \dots T_{\al_{t-1}}^{-1}F_{\al_t})^{n_t} \,\dots\,(T_{\al_1}^{-1}F_{\al_2})^{n_2}     F_{\al_1}^{n_1} K_\mu \,|\, n_1,n_2,\dots,n_t\in \N_0, \mu\in Q\}
  \end{align*}
  is a basis of $S(U^-[w])U^0$.
\end{cor}
\begin{proof}
Observe first that 
\begin{align*}
  T_w^{-1}(T_{\al_1} T_{\al_2}\dots T_{\al_{j-1}}E_{\al _j})
  =- T_{\al_t}^{-1} T_{\al_{t-1}}^{-1}\dots T_{\al_{j+1}}^{-1}
  (K_{\al _j}^{-1}F_{\al _j}).
\end{align*}
 Hence, in view of the PBW-basis of $U^+[w]$, Lemma \ref{lem:S-rewrite} implies that the elements
 \begin{align*}
   (T_{\al_t}^{-1} \dots T_{\al_{2}}^{-1}F_{\al_1})^{n_t}\, \dots\,  (T_{\al_t}^{-1}F_{\al_{t-1}})^{n_2} \,F_{\al_t}^{n_1} K_\mu
 \end{align*}
  for $(n_1,\dots,n_t)\in \N_0^t$ and $\mu\in Q$ form a basis of
  $S(U^-[w^{-1}])U^0$. Now replace $w$ by $w^{-1}$ to obtain the claim of the
  corollary.
\end{proof}
  Lemma \ref{lem:S-rewrite} moreover allows us to translate Corollary \ref{cor:U+wdec} into the setting of $U^-$.
\begin{cor} \label{cor:SU-wdec}
  Let $x,y\in W$ with $\ell (xy)=\ell (x)+\ell (y)$. Then
  $$S(U^-[xy])U^0=\big(T^{-1}_{x^{-1}}S(U^-[y])\big)S(U^-[x])U^0=S(U^-[x]) \big(T^{-1}_{x^{-1}}S(U^-[y])\big)U^0  .$$
\end{cor}

\section{Compatible pairs of coideal subalgebras}\label{sec:compatible}
\subsection{Classification}
We summarize the main results of the previous section.
\begin{prop} \label{prop:vw}
  \begin{enumerate}
     \item Any homogeneous right coideal subalgebra of $U$ is of the form $S(U^-[v])U^+[w]U^0$ for some $v,w\in W$.
  
     \item The pair $(v,w)$ in (1) is uniquely determined by the coideal subalgebra.
  
     \item For any $v,w\in W$ the subspace $S(U^-[v])U^+[w]U^0$ is  a right coideal subalgebra if and only if it is a subalgebra of $U$.
  \end{enumerate}   
\end{prop}
\begin{proof}
  Property (1) follows from the triangular decomposition given by Proposition \ref{prop:triang}, from the classification in Theorem \ref{thm:rcsa+}, and from Corollary \ref{cor:rcsa-}.
  Property (2) holds because $U^+[w]\neq U^+[w']$ if $w\neq w'$.
  Finally, Property (3) just expresses the fact that $S(U^-[v])U^+[w]U^0$
  is a right coideal which also holds by Theorem \ref{thm:rcsa+}.
\end{proof}
By the above proposition, to determine all homogeneous right coideal subalgebras of $U$ we need to determine the set
\begin{align*}
  A(W):=\{(v,w)\in W^2\,|\,S(U^-[v])U^+[w] U^0 \mbox{ is a subalgebra of $U$}\}.
\end{align*}
To this end we first provide a few preparatory lemmas.
\begin{lem}\label{lem:vw}
  Let $v,w\in W$. Then $(v,w)\in A(W)$ if and only if $(w,v)\in A(W)$.
\end{lem}
\begin{proof}
  Observe that $S\circ \omega$ is an algebra antiautomorphism such that $(S\circ \omega)^2=\id$. Moreover, $S\circ\omega(U^+[w])=S(U^-[w])$ and $S\circ \omega(S(U^-[v]))=\omega(U^-[v])=U^+[v]$. Hence $S(U^-[v]) U^+[w] U^0$ is an algebra if and only if 
\begin{align*}
  S\circ \omega(S(U^-[v]) U^+[w] U^0)=S(U^-[w]) U^+[v] U^0
\end{align*}
is an algebra.
\end{proof}
\begin{lem}\label{lem:rueber}
  Let $(v,w)\in W^2$ and $\alpha \in \Pi $ such that $\ell (s_\alpha v)=\ell (v)+1$
  and $\ell (s_\alpha w)=\ell (w)+1$. Then
  $(s_\alpha v,w)\in A(W)$ if and only if $(v,s_\alpha w)\in A(W)$.
\end{lem}
\begin{proof}
  As $T_\alpha $ is an algebra automorphism of $U$, one has $(s_\alpha
  v,w)\in A(W)$ if and only if
  $T_\alpha (S(U^-[s_\alpha v])U^+[w] U^0)$ is a subalgebra of $U$. Using
  Proposition~\ref{prop:PBW+} and
  Corollary~\ref{cor:SU-wdec} one obtains that
  \begin{align*}
    T_\alpha (S(U^-[s_\alpha v])U^+[w] U^0)=S(U^-[v])U^+[s_\alpha w]U^0
  \end{align*}
  which completes the proof of the lemma.
\end{proof}
\begin{lem}\label{lem:ik}
  Let $(v,w)\in W^2$ and $\alpha \in \Pi $ such that $\ell (s_\alpha v)=\ell (v)+1$
  and $\ell (s_\alpha w)=\ell (w)+1$. Assume that $(s_\alpha v,s_\alpha w)\in A(W)$.
  Then $\alpha \in v\Pi \cap w\Pi $.
\end{lem}
\begin{proof}
  The assumption $\ell (s_\alpha v)=\ell (v)+1$ implies that
  $F_\alpha\in S(U^-[s_\alpha v])U^0$.
  Hence $(s_\alpha v, s_\alpha w)\in A(W)$ implies that
  $\field\langle F_\alpha \rangle U^+[ s_\alpha w]U^0$ is a subalgebra
  of $U$. Hence 
  \begin{align*}
    T_\alpha ^{-1}(\field\langle F_\alpha \rangle U^+[s_\alpha w]U^0)
    =\field \langle E_\alpha ,F_\alpha \rangle U^+[w]U^0
  \end{align*}
  is a subalgebra of $U$. This is only possible if $\field \langle
  E_\alpha \rangle U^+[w]U^0$ is a subalgebra of $\uqbp$. Hence
  $\alpha \in w\Pi $ by Lemma~\ref{lem:EUw}.
  The statement about $v$ follows from the above argument and
  Lemma \ref{lem:vw}.
\end{proof}
For any subset $J\subseteq \Pi $ we write $w_J$ to denote the longest element
of the parabolic subgroup $W_J$ corresponding to $J$. For any $v\in W$ define
\begin{align*}
  J_v=\{\alpha \in \Pi\,|\,\ell (s_\alpha v)<\ell (v)\}.
\end{align*}
\begin{lem}\label{lem:vvJ}
  Let $v\in W$ and set $J=J_v$. Assume that for any $\alpha \in J$ there
  exists $\beta \in \Pi $ such that $vs_\beta =s_\alpha v$. Then $v=w_J$.
\end{lem}
\begin{proof}
  Lemma~\ref{lem:EUw}(2) with $w=s_\alpha v$ implies that
  $v^{-1}\alpha \in -\Pi $ for all $\alpha \in J$. Let $K=\{-v^{-1}\alpha
  \,|\,\alpha \in J\}$ and $u=w_Jv$.
  Then $K\subseteq \Pi $ and $u^{-1}\alpha \in \Pi $ for all $\alpha \in J$.
  Let $\beta \in \Pi \setminus J$ and let $\gamma \in \sum _{\alpha
  \in J}\N _0\alpha $ such that $w_J^{-1}\beta =\beta +\gamma $.
  Then
  $$u^{-1}\beta = v^{-1}(\beta +\gamma )\in v^{-1}\beta -\sum _{\alpha \in K}
  \N _0\alpha .$$
  Since $\beta \notin J=J_v$, we conclude that $v^{-1}\beta \in \Phi ^+$.
  Further, since $\beta \notin \sum _{\alpha \in J}\Z \alpha $,
  also $v^{-1}\beta \notin \sum _{\alpha \in K}\N _0\alpha $ holds.
  Thus $u^{-1}\beta \in \Phi ^+$. Therefore $u^{-1}\Pi \subseteq \Phi ^+$,
  that is, $\ell (u)=0$. Since $w_J=w_J^{-1}$, this implies that $v=w_J$.
\end{proof}
\begin{prop}\label{prop:decomp}
  Let $(v,w)\in A(W)$. Then there exist elements $u,x\in W$
  and a subset $J\subseteq \Pi $ such that the following properties hold:
  \begin{enumerate}
     \item $v=uw_J$ and $w=uw_J x$.
     \item $J\subseteq \Pi\cap x\Pi$.
     \item $u^{-1} \rwo x$.
  \end{enumerate}
  Moreover, the triple $(x,u,J)$ is uniquely determined by the pair $(v,w)$
  and one has
  \begin{align}
    \ell (v)=&\;\ell (u)+\ell (w_J), \label{eq:ellv} \\
    \ell (w)=&\;\ell (x)+\ell(w_J)-\ell (u). \label{eq:ellw}
  \end{align}
\end{prop}
\begin{proof}
  Choose $u\rwo v$ maximal such that $\ell (u^{-1} w)=\ell (u)+\ell (w)$.
  By Lemma \ref{lem:rueber} one has $(u^{-1}v,u^{-1}w)\in A(W)$. Define
  $J=J_{u^{-1}v}$. By choice of $u$ one has
  \begin{align}\label{eq:siJ}
   \ell (s_\alpha u^{-1}w)=\ell (u^{-1}w)-1 \qquad \mbox{ for any $\alpha \in J$.}
  \end{align}
  This and Lemma \ref{lem:ik} yield that
  $\alpha \in s_\alpha u^{-1}v\Pi $ for all $\alpha \in J$. By
  Lemma~\ref{lem:EUw}(2), for all $\alpha \in J$ there exists
  $\beta \in \Pi $ such that $u^{-1}v=s_\alpha u^{-1}vs_\beta $,
  that is, $s_\alpha u^{-1}v=u^{-1}vs_\beta $.
  Then Lemma \ref{lem:vvJ} implies that $u^{-1}v=w_J$. In particular,
  Equation~\eqref{eq:ellv} holds.
  Moreover, Equation \eqref{eq:siJ} implies that $w_J\rwo u^{-1}w$.
  Hence there exists $x\in W$ such that $u^{-1}w=w_Jx$ and
  $\ell (w_J)+\ell (x)=\ell (u^{-1}w)$. In particular, property (1) holds and
  Equation~\eqref{eq:ellw} is satisfied by the choice of $u$.
  
  To verify (2), choose $\alpha \in J$ and write $w_Jx=s_\alpha y$ with
  $\ell (w_Jx)=\ell (y)+1$. As $(w_J,w_Jx)\in A(W)$, Lemma \ref{lem:ik} implies
  that $\alpha \in y\Pi $. Hence $-\alpha \in w_Jx\Pi $, and hence
  $-w_J\alpha \in x\Pi $. Since $-w_J J=J$, property (2) holds.
 
  To prove (3), we first conclude from (2) and from Lemma~\ref{lem:EUw}(2)
  that there exists $K\subseteq \Pi $ with $xK=J$ and $w_Jx=xw_K$.
  By the choice of $u$ and $J$ we obtain that $\ell (us_\alpha )=\ell (u)+1$
  for all $\alpha \in J$ and hence $uJ\subseteq \Phi ^+$.
  Thus $uxK=uJ\subseteq \Phi ^+$,
  and therefore
  $$\ell (w)=\ell(uw_Jx)=\ell (uxw_K)=\ell (ux)+\ell (w_K)=\ell (ux)+\ell
  (w_J).$$
  Hence $\ell (ux)=\ell (x)-\ell (u)$ by Equation~\eqref{eq:ellw},
  that is, property (3) holds.

To prove uniqueness of the triple $(x,u,J)$,
observe first that $x=v^{-1}w$ is uniquely determined by $v$ and $w$.
Define $M=\Pi \cap x\Pi $. Then $J\subseteq M$ by (2). Moreover,
$\ell (x^{-1}s_\alpha )=\ell (x^{-1})+1$ for all $\alpha \in M$.
Hence, $x^{-1}$ is a minimal length left coset representative of
$x^{-1}W_M\in W/W_M$.
Property (3) implies that $u$ is a minimal
length left coset representative of $uW_M\in W/W_M$.
Therefore $u$ and $w_J$ are
uniquely determined by $v=uw_J$, see \cite[Proposition 2.4.4]{b-BjornerBrenti06}.
\end{proof}
Let $P(\Pi )$ denote the power set of $\Pi $.
Motivated by the above proposition we define a subset $B(W)\subseteq W^2\times
P(\Pi )$ by
\begin{align*}
  B(W)=\{(x,u,J)\in W^2\times P(\Pi )\,|\, J\subseteq \Pi\cap x\Pi
  \mbox{ and } \,u^{-1} \rwo x\}.
\end{align*}
We now show that we have a well defined map from $B(W)$ to $A(W)$.
\begin{prop}\label{prop:well-def}
  If $(x,u,J)\in B(W)$ then $(uw_J, uw_J x)\in A(W)$.
\end{prop}
\begin{proof}
  Let $(x,u,J)\in B(W)$. As $u^{-1}\rwo x$ and $\ell (s_\alpha x)=\ell (x)+1$
  for all $\alpha \in J$ one has $\ell (w_Jx)=\ell (w_J)+\ell (x)$ and
  $\ell (u w_J)=\ell (u) + \ell (w_J)$. Moreover,
  $J\subseteq \Pi \cap x\Pi $ and hence for all
  $\alpha \in J$ there exists $\beta \in \Pi $ with
  $s_\alpha x=xs_\beta $ by Lemma~\ref{lem:EUw}(2).
  This implies that $w_J x=x w_K$ for some subset $K\subseteq \Pi $.
  Hence $uw_J x=uxw_K$. Together with $\ell (w_J x)=\ell (w_J)+\ell (x)
  =\ell (x)+\ell (w_K)$
  this implies that
  $$\ell (uxw_K)= \ell (x)-\ell (u)+\ell (w_J)=\ell (w_J x) - \ell (u).$$
  Hence we can apply Lemma \ref{lem:rueber} and it suffices to show that
  $(w_J,w_J x)\in A(W)$.
  
Now, since $w_Jx=xw_K$ and $xK=J$, we conclude that
$T_x(U^+[w_K])=U^+[w_J]$ and hence Corollary \ref{cor:U+wdec} implies that
\begin{align}\label{eq:step}
  U^+[w_J x]=U^+[x] U^+[w_J]. 
\end{align}

To prove the proposition note that $U^+[w_J] T_{w_J}(U^+[x])= U^+[w_J x]$ is
an algebra. Hence Lemma \ref{lem:S-rewrite} implies that $T_{w_J}(S(U^-[w_J])
U^+[x] U^0)$ and hence also $S(U^-[w_J]) U^+[x] U^0$ are algebras. Therefore
$S(U^-[w_J]) U^+[x] U^+[w_J] U^0$ is an algebra. By
\eqref{eq:step} this proves that $(w_J,w_J x)\in A(W)$.
\end{proof}
The following theorem summarizes Propositions \ref{prop:decomp} and
\ref{prop:well-def}. It states that the set $B(W)$ parametrizes the set of homogeneous
right coideal subalgebras of $U$.
\begin{thm} \label{thm:rcsaU}
  The map $B(W)\rightarrow A(W)$ given by $(x,u,J)\mapsto (uw_J, uw_J x)$ is a
  bijection.
\end{thm}
For any subset $J\subseteq \Pi $ let $U_q(\mathfrak{k}_J)$ denote the Hopf
subalgebra of $\uqg $ generated by $U^0$ and by $E_\alpha ,F_\alpha $ with
$\alpha \in J$.  By Propositions~\ref{prop:vw} and \ref{prop:well-def}
with $u=e$, for any $x\in W$ and any $J\subseteq \Pi \cap x\Pi $ the subspace
$S(U^-[w_J])U^+[w_Jx]U^0$ is a homogeneous right coideal subalgebra of $\uqg $.
By Equation~\eqref{eq:step} this is equivalent to 
$U_q(\mathfrak{k}_J)U^+[x]$ being a homogeneous right coideal subalgebra
of $\uqg $.
\begin{cor}
  For any homogeneous right coideal subalgebra $C$ of $\uqg $
  there exist $x\in W$ and $J\subseteq \Pi \cap x\Pi $
  such that $C$ is isomorphic as an algebra
  to $U_q(\mathfrak{k}_J)U^+[x]$.
\end{cor}

\begin{proof}
 By Theorem~\ref{thm:rcsaU}, any homogeneous right coideal subalgebra
 of $\uqg $ is of the form $S(U^-[uw_J])U^+[uw_Jx]U^0$ for some $x,u\in W$
 and $J\subseteq \Pi \cap x\Pi $, where $u^{-1}\rwo x$.
 Since $u\rwo uw_J$, we conclude from Corollary~\ref{cor:SU-wdec} that
 $$S(U^-[uw_J])U^0=T^{-1}_{u^{-1}}(S(U^-[w_J]))S(U^-[u])U^0.$$
 Therefore
 \begin{align*}
   &T_{u^{-1}}\big(S(U^-[uw_J])U^+[uw_Jx]U^0\big)\\
   &\quad =T_{u^{-1}}\big( T_{u^{-1}}^{-1}(S(U^-[w_J]))S(U^-[u])U^0
   U^+[uw_Jx]\big)\\
   &\quad =S(U^-[w_J])U^+[u^{-1}]U^0 T_{u^{-1}}(U^+[uw_Jx])\\
   &\quad =S(U^-[w_J])U^+[w_Jx]U^0
 \end{align*}
 where the second equation holds by Lemma~\ref{lem:S-rewrite}
 and the third by Corollary~\ref{cor:U+wdec}
 since $\ell (uw_Jx)=\ell (w_Jx)-\ell (u)$ by
 Equation~\eqref{eq:ellw}.
 The last expression equals $U_q(\mathfrak{k}_J)U^+[x]$ by
 Equation~\eqref{eq:step}.
\end{proof}

\subsection{Calculations}
By Theorem \ref{thm:rcsaU} the number of homogeneous right coideal subalgebras of $U$ coincides with the cardinality $|B(W)|$ of the set $B(W)$. The calculation of $|B(W)|$ is straightforward. For any $x\in W$ one has to determine the cardinality of the set $\{u\in W\,|\,u^{-1}\le_R x\}$ and multiply by $2^{|\Pi \cap x\Pi|}$ to account for the choice of $J$. Finally one has to sum over all $x\in W$. In formulas, this gives
\begin{align}\label{eq:BW}
  |B(W)|=\sum_{x\in W} |\{u\in W\,|\,u\le _R x\}| \cdot 2^{|\Pi\cap x\Pi|}.
\end{align}
As $\Pi\cap x\Pi=\{\beta\in \Pi\,|\, \exists \alpha \in \Pi \mbox{ such that } s_\beta=x s_\alpha x^{-1}, \ell(x s_\al) = \ell(x) +1 \}$, the cardinality $|B(W)|$ is given purely in terms of the Weyl group. This proves the following result.
\begin{cor}
  The number of homogeneous right coideal subalgebras of $U $ depends only on
  the Weyl group $W$ of  $\gfrak $ as a Coxeter group but not explicitly on
  $\gfrak $.
\end{cor}
\begin{eg}\label{eg:G2}
  Let $\gfrak$ be of type $G_2$. For simplicity we denote the standard generators of $W$ by $s_1$ and $s_2$.
  In Table \ref{tab:G2} we have recorded the numbers $|\{u\in W\,|\,u\le_R x\}|$ and $|\Pi\cap x\Pi|$ for all
  $x\in W$. A straightforward calculation using \eqref{eq:BW} then yields $|B(W)|=68$. This calculation
  corrects the number given in \cite{a-Pogorelsky11}, see also Remark~\ref{rem:diffnumbers}.
  \begin{table}
  \centering
  \begin{tabular}{|c|c|c|}
    \hline
    $x$ & $|\{u\in W\,|\,u\le_R x\}|$ & $|\Pi\cap x\Pi|$ \\
    \hline
    $e$ & $1$ & $2$ \\
    $s_1$ & $2$ & $0$ \\
    $s_2$ & $2$ & $0$ \\
    $s_1 s_2$ & $3$ & $0$ \\
    $s_2 s_1$ & $3$ & $0$ \\
    $s_1 s_2 s_1$ & $4$ & $0$ \\
    $s_2 s_1 s_2$ & $4$ & $0$ \\
    $s_1 s_2 s_1 s_2$ & $5$ & $0$ \\
    $s_2 s_1 s_2 s_1$ & $5$ & $0$ \\
    $s_1 s_2 s_1 s_2 s_1$ & $6$ & $1$ \\
    $s_2 s_1 s_2 s_1 s_2$ & $6$ & $1$ \\
    $s_1 s_2 s_1 s_2 s_1 s_2$ & $12$ & $0$ \\
    \hline
  \end{tabular}
  \caption{Calculating $|B(W)|$ in the case $G_2$}
  \label{tab:G2}
\end{table}
\end{eg}
It is possible to perform calculations as in Example \ref{eg:G2} by hand
for $W$ of types $A_2$, $A_3$, $B_2$, and $B_3$. To obtain the number $|B(W)|$
also for examples with larger Weyl groups we have used the computer algebra
program FELIX \cite{inp-ApelKlaus,a-FELIX}. This allowed us to determine
$|B(W)|$ for all Weyl groups of cardinality less than $10^6$. The results are
displayed in Table~\ref{tab:nrcsa}. Our calculations confirm the results
obtained in \cite{a-KharSag08} for type $A_n$.

\begin{rema} \label{rem:diffnumbers}
  In \cite{a-Pogorelsky11}, Pogorelsky applies Kharchenko's method to determine the
  homogeneous right coideal subalgebras of $\uqg $ and of the related small quantum group
  for $\gfrak $ of type $G_2$.
  She obtains a number different from ours due to a hidden mistake in her previous work \cite{a-Pogorelsky09}
  on right coideal subalgebras of the standard Borel Hopf subalgebra of $\uqg $.

  In \cite{a-KharSagRiv11}, Kharchenko, Sagahon and Rivera calculate the number of homogeneous
  right coideal subalgebras of $\uqg $ for $\gfrak $ of type $B_n$ with $n\ge 2$.
  Again, the numbers for $n\ge 3$ are different from (more precisely, smaller than) the numbers given in 
  Table~\ref{tab:nrcsa}. We have checked the cases $B_3$ and $B_4$ and confirm that both our description and the theory in
  \cite{a-KharSagRiv11} provide the same set of right coideal subalgebras and consequently the same numbers, that is,
  $664$ for $B_3$ and $17848$ for $B_4$. Hence there seems to be an error in the explicit computer calculations in \cite{a-KharSagRiv11}.  
\end{rema}

\begin{table}
  \centering
  \begin{tabular}{|r|r|r|}
    \hline
    $W$ & $|W|$ & $|B(W)|$ \\
    \hline
    $A_1$ & $2$ & $4$ \\
    $A_2$ & $6$ & $26$ \\
    $A_3$ & $24$ & $252$ \\
    $A_4$ & $120$ & $3368$ \\
    $A_5$ & $720$ & $58810$ \\
    $A_6$ & $5040$ & $1290930$ \\
    $A_7$ & $40320$ & $34604844$ \\
    $A_8$ & $362880$ & $1107490596$ \\
    \hline
    $B_2$ & $8$ & $38$ \\
    $B_3$ & $48$ & $664$ \\
    $B_4$ & $384$ & $17848$ \\
    $B_5$ & $3840$ & $672004$ \\
    $B_6$ & $46080$ & $33369560$ \\
    $B_7$ & $645120$ & $2094849020$ \\
    \hline
    $D_4$ & $192$ & $6512$ \\
    $D_5$ & $1920$ & $238720$ \\
    $D_6$ & $23040$ & $11633624$ \\
    $D_7$ & $322560$ & $720453984$ \\
    \hline
    $E_6$ & $51840$ & $38305190$ \\
    \hline
    $F_4$ & $1152$ & $91244$ \\
    \hline
    $G_2$ & $12$ & $68$\\
    \hline
  \end{tabular}
  \caption{Number of homogeneous right coideal subalgebras of $U $}
  \label{tab:nrcsa}
\end{table}


\begin{thebibliography}{CKP95}

\bibitem[AK]{a-FELIX}
J.~Apel and U.~Klaus, \emph{{F}{E}{L}{I}{X}},
  \verb+(http://felix.hgb-leipzig.de)+.

\bibitem[AK91]{inp-ApelKlaus}
\bysame, \emph{{FELIX} -- an assistant for algebraists}, ISSAC'91 (New York)
  (S.M. Watt, ed.), ACM Press, 1991, pp.~382--389.

\bibitem[BB06]{b-BjornerBrenti06}
A.~Bj{\"o}rner and F.~Brenti, \emph{Combinatorics of {C}oxeter groups},
  Springer-Verlag, New York, 2006.

\bibitem[CKP95]{a-DCoKacPro95}
C.~De Concini, V.G. Kac, and C.~Procesi, \emph{Some quantum analogues of
  solvable {L}ie groups}, Geometry and analysis (Bombay 1992) (Bombay), Tata
  Inst. Fund. Res., 1995, pp.~41--65.

\bibitem[Dij96]{a-Dijk96}
M.S. Dijkhuizen, \emph{Some remarks on the construction of quantum symmetric
  spaces}, Acta Appl. Math. \textbf{44} (1996), no.~1-2, 59--80.

\bibitem[HK11]{a-heko11}
I.~Heckenberger and S.~Kolb, \emph{Right coideal subalgebras of the {B}orel
  part of a quantized enveloping algebra}, Int. Math. Res. Not. IMRN
  \textbf{2011} (2011), no.~2, 419--451.

\bibitem[HS09]{a-HeckSchn09p}
I.~Heckenberger and H.-J. Schneider, \emph{Right coideal subalgebras of
  {N}ichols algebras and the {D}uflo order on the {W}eyl groupoid}, Preprint,
  {\ttfamily arXiv:0909.0293} (2009), 43 pp.

\bibitem[Jan96]{b-Jantzen96}
J.C. Jantzen, \emph{Lectures on quantum groups}, Grad. Stud. Math., vol.~6,
  Amer. Math. Soc, Providence, RI, 1996.

\bibitem[Kha10]{a-Kharchenko10}
V.~K. Kharchenko, \emph{Triangular decomposition of right coideal subalgebras},
  J. Algebra \textbf{324} (2010), no.~11, 3048--3089.

\bibitem[Kha11]{a-Kharchenko09}
\bysame, \emph{Right coideal subalgebras in ${U}_q^+(\mathfrak{so}_{2n+1})$},
  J. Eur. Math. Soc. (JEMS) \textbf{13} (2011), 1675--1733.

\bibitem[KS08]{a-KharSag08}
V.~K. Kharchenko and A.~V.~Lara Sagahon, \emph{Right coideal subalgebras in
  ${U}_q(\mathfrak{sl}_{n+1})$}, J. Algebra \textbf{319} (2008), no.~6,
  2571--2625.

\bibitem[KSR11]{a-KharSagRiv11}
V.~K. Kharchenko, A.~V.~Lara Sagahon, and J.~L.~Garza Rivera, \emph{Computing
  of the number of right coideal subalgebras of ${U}_q(\mathfrak{so}_{2n+1})$},
  J. Algebra \textbf{341} (2011), no.~1, 279--296.

\bibitem[Let02]{MSRI-Letzter}
G.~Letzter, \emph{Coideal subalgebras and quantum symmetric pairs}, New
  directions in {H}opf algebras (Cambridge), MSRI publications, vol.~43,
  Cambridge Univ. Press, 2002, pp.~117--166.

\bibitem[NS95]{a-NS95}
M.~Noumi and T.~Sugitani, \emph{Quantum symmetric spaces and related
  $q$-orthogonal polynomials}, Group theoretical methods in physics (Singapore)
  (A.~Arima et. al., ed.), World Scientific, 1995, pp.~28--40.

\bibitem[Pog09]{a-Pogorelsky09}
B.~Pogorelsky, \emph{Right coideal subalgebras of the quantum {B}orel algebra
  of type ${G}_2$}, J. Algebra \textbf{322} (2009), no.~7, 2335--2354.

\bibitem[Pog11]{a-Pogorelsky11}
\bysame, \emph{Right coideal subalgebras of quantized universal enveloping
  algebras of type ${G}_2$}, Comm. Algebra \textbf{139} (2011), no.~4,
  1181--1207.

\bibitem[Yak10]{a-Yakimov10}
M.~Yakimov, \emph{Invariant prime ideals in quantizations of nilpotent {L}ie
  algebras}, Proc. London Math. Soc. \textbf{101} (2010), no.~2, 454--476.

\end{thebibliography}

\providecommand{\bysame}{\leavevmode\hbox to3em{\hrulefill}\thinspace}
\providecommand{\MR}{\relax\ifhmode\unskip\space\fi MR }
\providecommand{\MRhref}[2]{%
  \href{http://www.ams.org/mathscinet-getitem?mr=#1}{#2}
}
\providecommand{\href}[2]{#2}

\end{document}